\documentclass[12pt,twoside]{amsart}

\usepackage{amsmath,amsthm,amscd,amssymb,mathrsfs,graphicx,amsfonts,mathrsfs}
\usepackage{amsfonts}
\usepackage{amssymb,enumerate}
\usepackage{amsthm}
\usepackage[all]{xy}
\usepackage{hyperref}
\allowdisplaybreaks[4] \footskip=12pt
\renewcommand{\uppercasenonmath}[1]{}

\numberwithin{equation}{section} \theoremstyle{plain}
\newtheorem*{thm*}{Main Theorem}
\newtheorem{thm}{Theorem}[section]
\newtheorem{cor}[thm]{Corollary}
\newtheorem*{cor*}{Corollary}
\newtheorem{lem}[thm]{Lemma}
\newtheorem*{lem*}{Lemma}

\newtheorem*{fact*}{Fact}

\newtheorem*{nota*}{Notation}
\newtheorem{prop}[thm]{Proposition}
\newtheorem*{prop*}{Proposition}
\newtheorem{rem}[thm]{Remark}
\newtheorem*{rem*}{Remark}

\newtheorem*{observation*}{Observation}

\newtheorem*{exa*}{Example}
\newtheorem{df}[thm]{Definition}
\newtheorem*{df*}{Definition}

\newtheorem*{con*}{Construction}

\renewcommand{\geq}{\geqslant}
\renewcommand{\leq}{\leqslant}

\begin{document}
\begin{center}
{\large  \bf Quasi-projective dimensions of complexes over rings}

\vspace{0.5cm} Hongxing Chen, Jiangsheng Hu and Xiaoyan Yang\\
\end{center}

\bigskip
\centerline { \bf  Abstract}
\leftskip10truemm \rightskip10truemm\noindent
Quasi-projective dimension of modules over associative rings is generalized in this paper to the one of complexes of modules. Basic properties of this dimension are established, including a comparison result with projective dimension and a derived Auslander-Buchsbaum formula for complexes of finite quasi-projective dimension. Several sufficient conditions are provided for a commutative noetherian local ring to be a complete intersection under the assumption that each finitely generated module has finite quasi-projective dimension. This provides some positive answers to an open question on quasi-projective dimension proposed by Gheibi-Jorgensen-Takahashi. Moreover, the behavior of quasi-projective dimension under taking the quotient of a commutative ring modulo a regular sequence is investigated, and some partial results toward the change-of-rings question on quasi-projective dimension are given.

\leftskip10truemm \rightskip10truemm \noindent
\vbox to 0.3cm{}

\leftskip0truemm \rightskip0truemm

\section{\bf Introduction}

Homological dimensions, such as projective dimension, G-dimension and complete intersection dimension, have long been fundamental invariants in the study of modules and complexes over commutative noetherian rings. They can be applied to measure homological complexity and detect the singularity of underlying rings, and play a central role in the structure or classification of derived and singularity categories. A classical example is the Auslander-Buchsbaum-Serre theorem on regularity which says that a commutative noetherian local ring is regular precisely when every finitely generated module over it has finite projective dimension.

Recently, Gheibi, Jorgensen and Takahashi introduced a new homological dimension for modules over associative rings, called \emph{quasi-projective dimension}, and established both the Auslander-Buchsbaum formula and the depth formula in commutative algebra for modules of finite quasi-projective dimension (see \cite{GJT}). This dimension generalizes projective dimension, but behaves differently from other homological dimensions; for instance, the category of finitely generated modules of finite quasi-projective dimension does not have two out of three property. Its importance is further underscored by its relation to the validity of the Auslander-Reiten conjecture for commutative noetherian rings (see \cite[Theorem 6.20]{GJT} and \cite[Corollary 1.3]{CCL}), and of the Tachikawa's second conjecture for finite-dimensional self-injective algebras (see \cite[Corollary 4.3]{CCL}). For more advances on these conjectures, we refer to \cite{ADS,CH10,HSV,CFX,CX}.

Let $R$ be an associative ring with identity. Following \cite[Definition 3.1]{GJT}, we say that a left $R$-module $M$ \emph{has finite quasi-projective dimension}, denoted by $\mathrm{qpd}_RM<\infty$, if there exists a bounded complex of projective modules whose homologies (not all zero) are isomorphic to finite direct sums of copies of $M$. By \cite[Corollary 3.8]{GJT}, if $R$ is a complete intersection and isomorphic to its completion with respect to its maximal ideal, then $\mathrm{qpd}_RM<\infty$ for any finitely generated $R$-module $M$. For the converse of this result, Gheibi, Jorgensen and Takahashi proposed an open question as follows.

\vspace{2mm} \noindent{\bf Question 1} \rm{\cite[Question 3.12]{GJT}.\label{Th1.4}
\emph{ Let $R$ be a commutative noetherian local ring. Suppose that $\mathrm{qpd}_RM<\infty$ for any finitely generated $R$-module $M$. Then is $R$ a complete intersection?}

\vspace{2mm}
This question is closely related to  virtually small complexes introduced by Dwyer, Greenlees and Iyengar in \cite{DGI}. Recall that a nonexact complex $M$ over $R$ is \emph{virtually small} if the thick subcategory of the derived category $\mathrm{D}^\mathrm{f}(R)$ (consisting of the complexes of $R$-modules with \emph{finitely generated homology}) generated by $M$ contains a nonexact perfect complex. In \cite[Theorem 5.2]{P}, Pollitz proved a deep structural theorem: a local ring $R$ is a complete intersection if and only if every nonexact complex in $\mathrm{D}^\mathrm{f}(R)$ is virtually small. However, verifying virtual smallness is difficult in general. Fortunately, modules of finite quasi-projective dimension are virtually small. Based on these observations, Briggs, Grifo and Pollitz \cite{BGP} gave an affirmative answer to Question 1 when $R$ is an equipresented local ring (see \cite[Corollary 3.17]{BGP}). Nonetheless, Question 1 is still open for a general local ring.

Another basic problem concerns the behavior of quasi-projective dimension under taking the quotient of a commutative ring modulo a regular element.

\vspace{2mm} \noindent{\bf Question 2} {\rm\cite[Question 2.12]{G}}. \label{Th1.4}
{\it{Let $R$ be a commutative noetherian ring, $x\in R$ a regular element and $M$ an $R/(x)$-module. If $\mathrm{qpd}_{R}M<\infty$, then is it true that $\mathrm{qpd}_{R/(x)}M<\infty$?}}

\vspace{2mm}
The posing of Question 2 is primarily motivated by the following results.
It was shown in \cite[Proposition 3.5]{GJT} that $\mathrm{qpd}_{R}M\leq\mathrm{qpd}_{R/(x)}M+1$, and in \cite[Proposition 2.11]{G} that if $\mathrm{qpd}_{R}M<\infty$ then $\mathrm{qpd}_{R/(x^n)}M<\infty$ for all $n\gg0$. Thus if Question 2 has a positive answer, then $\mathrm{qpd}_{R}M<\infty$ if and only if $\mathrm{qpd}_{R/(x)}M<\infty$.

Inspired by the connections among complete intersection rings, modules of finite quasi-projective dimension
and virtually small complexes, we develop a derived-level theory of quasi-projective dimension for \emph{complexes} over arbitrary rings, establish some basic properties of this dimension and provide some partial answers to Questions 1 and 2. To state our main results clearly, let us fix some notation on derived categories.

Let $\mathrm{D}(R)$ denote the unbounded derived category of complexes of left $R$-modules (called $R$-complexes for short), and $\mathrm{D}^\mathrm{f}(R)$ the full subcategory of $\mathrm{D}(R)$ consisting of $R$-complexes with finitely generated homology.  We also define $\mathrm{D}_\sqsubset(R)$ (resp., $\mathrm{D}_\sqsupset(R)$, $\mathrm{D}_\square(R)$) to be the full subcategory of $\mathrm{D}(R)$ consisting of homology bounded above (resp., bounded below, bounded) $R$-complexes and set $\mathrm{D}^\mathrm{f}_\Box(R):=\mathrm{D}_\square(R)\cap\mathrm{D}^\mathrm{f}(R)$.
For each $M\in\mathrm{D}(R)$, we denote by $\mathrm{pd}_RM$ and $\mathrm{qpd}_RM$ the projective and quasi-projective dimensions of $M$ (see Definition \ref{lem:1.1}), respectively. As usual, $\mathrm{H}(M):=\bigoplus_{n\in\mathbb{Z}}\mathrm{H}_n(M)$ denotes the homology of $M$.

Our first result establishes some relationships between quasi-projective and projective dimensions for complexes of modules, and generalizes some known formulas at the level of modules. For unexplained notation in the theorem, we refer to Section \ref{Preliminaries}.

\begin{thm}\label{thm:1.1} {\rm (see Theorem \ref{lem0.6}, Proposition \ref{lem0.14} and Theorem \ref{lem0.3})} {\it{ Let $R$ be an associative ring with identity. Then the following statements are true.

$(1)$ Let \(0\not\simeq M\in\mathrm{D}(R)\) with \(\mathrm{pd}_RM<\infty\), then
\(\mathrm{qpd}_R M=\mathrm{pd}_RM.\)

$(2)$ Let \(M\in\mathrm{D}_\Box(R)\) with $\mathrm{qpd}_RM<\infty$. If $\mathrm{Ext}^{j}_R(\mathrm{H}_s(M),\mathrm{H}_t(M))=0$ for all $j\geq2$ and $s,t\in\mathbb{Z}$, then $\mathrm{pd}_RM<\infty$.

$(3)$ $\mathrm{qpd}_R(M\oplus N)\leq\mathrm{max}\{\mathrm{qpd}_RM,\mathrm{qpd}_RN\}$ for $M,N\in\mathrm{D}(R)$ with $\mathrm{hsup}M=\mathrm{hsup}N$.

$(4)$ {Let $R$ be a commutative noetherian local ring and $M\in\mathrm{D}^\mathrm{f}_\Box(R)$ with $\mathrm{qpd}_RM<\infty$. Put $h=\operatorname{hsup}M$. If either $\mathrm{pd}_RM<\infty$, or $\mathrm{depth}_R\mathrm{H}_h(M)\leq\mathrm{depth}_R\mathrm{H}_i(M)$ for all $i\leq h$, then}
\begin{center}
$\mathrm{qpd}_RM=\mathrm{depth}R-\mathrm{depth}_RM.$
\end{center}}}
\end{thm}

Now, we compare Theorem \ref{thm:1.1} with some results in the literature.

Parts (1) and (2) of Theorem \ref{thm:1.1} were first established for finitely generated $R$-modules by Gheibi, Jorgensen and Takahashi (see \cite[Corollary 4.10]{GJT}). They were later generalized to objects in an abelian category with enough projectives (see \cite[Theorem 1.2]{CCL}) using derived category and derived functor. As $\mathrm{H}(M)\cong\bigoplus\mathrm{H}_i(M)$, Part (3) of Theorem \ref{thm:1.1}, when combined with \cite[Theorem 5.2]{P}, we see that Question 1 holds for any commutative noetherian local ring $R$ such that the homology $\mathrm{H}(M)$ of any non-perfect $R$-complex $M$ is finitely built from $M$ in $\mathrm{D}(R)$ (see Theorem \ref{lem3.004}(1)). Finally, part (4) provides a derived Auslander-Buchsbaum formula for $R$-complexes of finite quasi-projective dimension, refining the module-level formula in \cite[Theorem 1.1]{GJT}.

Matsui \cite{Mat} introduces a filtration-based quasi-projective
dimension and proves Auslander--Buchsbaum, derived depth and width
formulas. {Both constructions recover the module invariant, but our
definition uses identifying resolutions whereas his uses finite
filtrations in the derived category. His Auslander--Buchsbaum formula
requires no additional homology-depth hypothesis, whereas our formula
in Theorem~\ref{thm:1.1}(4) is established under either finite projective
dimension or the stated homology-depth condition.}

Next, we concentrate on Question 1 and provide some sufficient conditions for a commutative local ring to be a complete intersection by assuming that each finitely generated module has finite quasi-projective dimension.  Recall from \cite[Section 2]{DKT} that an ideal $I$ of a commutative noetherian local ring $(R,\mathfrak{m})$ is called a \emph{Burch ideal} if $\mathfrak{m}I\neq \mathfrak{m}(I:_R\mathfrak{m}),$ where $(I:_R\mathfrak{m}):=\{a\in R\hspace{0.03cm}|\hspace{0.03cm}\mathfrak{m}a\subseteq I\}$.

\begin{thm}\label{thm:1.2} {\rm(see Theorem \ref{lem3.004})} {\it{Let $(R,\mathfrak{m})$ be a commutative noetherian local ring such that its completion $\widehat{R}$ (with respect to the maximal ideal $\mathfrak{m}$) is isomorphic to $Q/I$, where $Q$ is a regular local ring with an ideal $I$.  Suppose that one of the following conditions holds:

$(1)$ $\mathrm{dim}Q\leq \mathrm{dim}R+2$.

$(2)$ $I$ is a Burch ideal of $Q$.

$(3)$ the minimal number of generators of $I$ is at most $2$.

$(4)$ the maximal ideal $\mathfrak{m}$ is decomposable.

\noindent If $\mathrm{qpd}_{R}M<\infty$ for every finitely generated $R$-module $M$, then $R$ is a complete intersection.}}
\end{thm}

A combination of Theorems \ref{thm:1.1} and \ref{thm:1.2} yields the following result.
\begin{cor} Let $R$ be an artinian ring that is isomorphic to $Q/I$, where $(Q,\mathfrak{n})$ is a regular local ring.
 If either the embedding dimension of $R$ is at most $2$, or $I$ is a Burch ideal of $Q$, or the minimal number of generators of $I$ is at most $2$, then $R$ is a complete intersection if and only if $\mathrm{qpd}_{R}M<\infty$ for all finitely generated $R$-modules $M$.
\end{cor}

Finally, we address Question 2 by establishing a series of reduction formulas of inequalities for quasi-projective dimension of complexes. The following theorem characterizes how quasi-projective dimension behaves when passing to the quotient of a ring modulo a regular sequence.

\begin{thm}\label{thm:1.3} {\rm(see Theorem \ref{lem3.002})} {\it{Let $R$ be a commutative ring with an $R$-regular sequence $\textbf{x}=x_1,\cdots,x_d$, and let $M$ be a nonexact complex in $\mathrm{D}_\Box(R/(\textbf{x}))$ with $\mathrm{qpd}_{R}M<\infty$. If $\mathrm{Ext}^{1}_R(\mathrm{H}_s(M),\mathrm{H}_t(M))=0$ for $s,t\in\mathbb{Z}$,
then $\mathrm{qpd}_{R/(\textbf{x})}M\leq\mathrm{qpd}_{R}M-d$.}}
\end{thm}

As a direct consequence of Theorem \ref{thm:1.3}, if $M$ is a nonzero, finitely generated module over a commutative noetherian local ring $R$ with $\mathrm{Ext}^{1}_R(M,M)=0$, then $\mathrm{qpd}_{R}M<\infty$ if and only if  {\it$\mathrm{qpd}_{R/(\textbf{x})}M<\infty$} (see Corollary \ref{lem4.004}). This offers a partial answer to Question 2.

The paper is organized as follows. In Section \ref{Preliminaries}, we fix some notations and recall some key definitions. In Section \ref{Quasi-projective-Properties}, we introduce the notion of quasi-projective dimension of complexes over rings and establish some basic properties for this dimension, including comparison with projective dimension and technical tools for resolutions. In particular, parts (1)-(3) of Theorem \ref{thm:1.1} are shown in this section. In Section \ref{AB-formula}, we prove a derived Auslander-Buchsbaum formula, that is, Theorem \ref{thm:1.1}(4), for complexes of finite quasi-projective dimension. In Section \ref{partial-ans-Q1}, we show Theorem \ref{thm:1.2}. In Section \ref{par-ans-Q2}, we discuss the behavior of quasi-projective dimension with respect to regular sequences and show Theorem \ref{thm:1.3}.

\bigskip
\section{\bf Preliminaries}\label{Preliminaries}
This section collects some notions and facts about complexes and commutative rings
for use throughout this paper.

\vspace{2mm}
1. {\bf Complexes}

\vspace{2mm}
We denote by $\mathrm{C}(R)$ the category whose class of objects consists of all
chain complexes $M$ of $R$-modules,
$$\cdots\longrightarrow M_{s+1}\stackrel{d_{s+1}^M}\longrightarrow M_s\stackrel{d_{s}^M}\longrightarrow M_{s-1}\longrightarrow\cdots$$such that $d_{s}^Md_{s+1}^M=0$ for all $s\in\mathbb{Z}$.
For $s\in\mathbb{Z}$, denote $$\mathrm{Z}_s(M)=\mathrm{Ker}d^M_{s},\ \mathrm{B}_{s}(M)=\mathrm{Im}d^M_{s+1},\ \mathrm{C}_{s}(M)=\mathrm{Coker}d^M_{s+1},\ \mathrm{H}_s(M)=\mathrm{Z}_s(M)/\mathrm{B}_s(M),$$ $$\mathrm{sup}M=\mathrm{sup}\{i\in\mathbb{Z}\hspace{0.03cm}|\hspace{0.03cm}M_i\neq0\},\ \mathrm{inf}M=\mathrm{inf}\{i\in\mathbb{Z}\hspace{0.03cm}|\hspace{0.03cm}M_i\neq0\},$$
$$\mathrm{hsup}M=\mathrm{sup}\{i\in\mathbb{Z}\hspace{0.03cm}|\hspace{0.03cm}\mathrm{H}_i(M)\neq0\},\
\mathrm{hinf}M=\mathrm{inf}\{i\in\mathbb{Z}\hspace{0.03cm}|\hspace{0.03cm}\mathrm{H}_i(M)\neq0\},$$
and $\mathrm{amp}(M)=\mathrm{hsup}M-\mathrm{hinf}M$. The $s$th \emph{shift} of $M$ is denoted by $\Sigma^sM$, i.e,
\((\Sigma^s M)_i = M_{i-s},\ d_i^{\Sigma^s M} = (-1)^s d_{i-s}^M.\)
Thus
\(\operatorname{hinf}(\Sigma^s M) = \operatorname{hinf} M + s\) and
\(\operatorname{hsup}(\Sigma^sM) = \operatorname{hsup} M + s.\)
Let $M$ and $N$ be $R$-complexes. Set $\textrm{Ext}^i_R(M,N)=\mathrm{H}_{-i}(\operatorname{RHom}_R(M,N))$.
The \emph{cone} of a chain map $f:M\rightarrow N$ is the complex with $\mathrm{cn}(f)_n=N_{n}\oplus M_{n-1}$ and
$d_n^{\mathrm{cn}(f)}(x,y)=(d_{n}^N(y)+f_{n-1}(x),-d_{n-1}^M(x))$ for all $n\in\mathbb{Z}$. A bounded complex of finitely generated projective $R$-modules is called \emph{perfect}. The class of perfect complexes is denoted by $\mathrm{perf}R$.

\vspace{2mm}
{\bf 1.1. Derived categories.} Denotes by $\mathrm{K}(R)$
the homotopy category of $R$-complexes and by $\mathrm{D}(R)$ the derived category of $R$-complexes. An isomorphism in $\mathrm{D}(R)$ is marked by $\simeq$. A quasi-isomorphism is marked by $\simeq$ next to the arrow.
 Write $M\otimes^\mathrm{L}_RN$ and  $\mathrm{RHom}_R(M,N)$
 for the derived functors of the tensor product functor and the homomorphism functor,
respectively. The existence of projective and injective resolutions ensure that derived
functors exist, and can be calculated with suitable resolutions of either factor. An $R$-complex $M$ is called \emph{minimal} if every isomorphism
$M\rightarrow M$ in $\mathrm{K}(R)$ is an isomorphism in $\mathrm{C}(R)$.

\vspace{2mm}
{\bf 1.2. Projective dimensions.} A morphism $\beta$ of complexes is a \emph{quasi-isomorphism} if
$\mathrm{H}(\beta)$ is bijective. An $R$-complex $P$ is called \emph{semi-projective} if $\mathrm{Hom}_R(P,\beta)$ is a
surjective quasi-isomorphism for every surjective quasi-isomorphism $\beta$ in $\mathrm{C}(R)$. A \emph{semi-projective resolution} of an $R$-complex $M$ is a
quasi-isomorphism $P\stackrel{\simeq}\rightarrow M$ of $R$-complexes with $P$ semi-projective. By \cite[Theorem 5.2.14]{CFH}, Every $R$-complex has a semi-projective resolution.
 The \emph{projective dimension} of an $R$-complex $M$, $\mathrm{pd}_RM$, is defined as
\begin{center}$\mathrm{pd}_RM=\mathrm{inf}\{n\in\mathbb{Z}\hspace{0.03cm}|\hspace{0.03cm}\exists\ \textrm{a\ semi-projective\ resolution}\ P\stackrel{\simeq}\rightarrow M\ \textrm{with}\ P_v=0\ \textrm{for\ all}\ v>n\}$.\end{center}
One has $\mathrm{pd}_R\Sigma^sM=\mathrm{pd}_RM+s$ for $s\in\mathbb{Z}$.

{\bf 1.3. Thick subcategories.} A thick subcategory $\mathcal{T}$ of $\mathrm{D}(R)$ is a nonempty full subcategory such that
 $\mathcal{T}$ is closed under isomorphisms and direct summands in $\mathrm{D}(R)$, and
 if two of the objects in an exact triangle in $\mathrm{D}(R)$ are in $\mathcal{T}$, then so is the third.
Let $M$ be an $R$-complex.
The intersection of thick subcategories of $\mathrm{D}(R)$ containing
$M$ is a thick subcategory, the thick subcategory \emph{generated} by $M$, and denote
it by $\mathrm{Thick}_R[M]$. It is suggestive to think of complexes in $\mathrm{Thick}_R[M]$ as being \emph{finitely built} from $M$. In general,
$\mathrm{Thick}_R[M]\subseteq\mathrm{Thick}_R[\mathrm{H}(M)]$.

\vspace{2mm}
2. {\bf Commutative rings}

\vspace{2mm}
Let $R$ be a commutative ring and $\mathrm{Spec}R$ the set of
prime ideals of $R$. The \emph{support} of an $R$-complex $M$ is defined as $\mathrm{Supp}_RM:=\{\mathfrak{p}\in\textrm{Spec}R\hspace{0.03cm}|\hspace{0.03cm}M_\mathfrak{p}\not\simeq 0\}$.

\vspace{2mm}
{\bf 2.1. Koszul complexes.} The complex $K(x)= 0\rightarrow R\stackrel{x}\rightarrow
R\rightarrow 0$ for $x\in R$,
concentrated in degrees 1 and 0, is called the \emph{Koszul complex} of $x$. The Koszul complex $K(\textbf{\emph{x}})=K(x_1,\cdots,x_n)$ of a finite sequence $\textbf{\emph{x}}=x_1,\cdots,x_n$ in $R$
is the tensor product $K(x_1)\otimes_R\cdots\otimes_RK(x_n)$.  For $M\in\mathrm{D}(R)$, we set $K(\textbf{\emph{x}};M)=K(\textbf{\emph{x}})\otimes_RM$. Let $\mathfrak{a}=(x_1,\cdots,x_n)$. Also denoted $K(\textbf{\emph{x}};M)$ by $K(\mathfrak{a};M)$. For a local noetherian ring $(R,\mathfrak{m})$, denote
$K(\mathfrak{m};R)$ the Koszul complex of a minimal system of generators
of $\mathfrak{m}$.

\vspace{2mm}
{\bf 2.2. Depth and dimension.} Let $(R,\mathfrak{m},k)$ be a local noetherian ring. The \emph{depth}
of an $R$-complex $M\in\mathrm{D}(R)$ is defined as follows
$$\text{\rm depth}_{R}M:=-\mathrm{hsupRHom}_{R}(k,M),$$
Note that for modules the notion agrees with the usual one.

 \vspace{2mm}
{\bf 2.3. Regular sequences.}  Like
most concepts for modules that of $M$-sequences can be
extended to complexes in several non-equivalent ways. Christensen \cite{C} explored
 regular sequences for complexes, which not only enjoys the localize properly, but also
agrees with the usual definition of $M$-regular
sequences for modules.

Let $0\not\simeq M\in \mathrm{D}_\sqsubset(R)$.
 Set $\mathrm{z}_R(M) = \mathrm{z}_R(\mathrm{H}_{\mathrm{hsup}M}(M))$ the set of zero-divisors for
$\mathrm{H}_{\mathrm{hsup}M}(M)$.
An element $x\in R$ is said to be \emph{regular} for $M$ if $x\not\in\mathrm{z}_R(M)$. Let $\textbf{\emph{x}}=x_1,\cdots,x_n$ be a sequence in $R$. We say that
$\textbf{\emph{x}}$ is an \emph{$M$-regular sequence} if $x_j$ is regular for
$K(x_1,\cdots,x_{j-1};M)$ for $j\in\{1,\cdots,n\}$, and $K(\textbf{\emph{x}};M)\not\simeq0$ or $M\simeq0$.
If $M$ is an $R$-module, then the definition of $M$-regular sequences agrees with the usual definition of $M$-regular sequences in \cite{BH}.

 Let $(R,\mathfrak{m},k)$ be a noetherian local ring and $\widehat{R}$ the $\mathfrak{m}$-adic completion of $R$.  By \cite[Theorem A.21]{BH}, $\widehat{R}\cong Q/I$, where $Q$ is a regular local ring and $I$ is an ideal of $Q$. The ring $R$ is called \emph{complete intersection} if $I$ is generated by a $Q$-regular sequence $(x_1,\cdots,x_n)$. If $n=1$, then $R$ is called a \emph{hypersurface}. Following \cite{T,T25}, $R$ is said to \emph{have an isolated singularity} if for every
nonmaximal prime ideal $\mathfrak{p}$ of $R$ the local ring $R_\mathfrak{p}$ is regular.
$R$ is called \emph{dominant} if $k\in \mathrm{Thick}_{R}[R,M]$ for each $M\in\mathrm{D}^\mathrm{f}_\Box(R)\backslash\mathrm{perf}R$.

\bigskip
\section{\bf Properties of quasi-projective dimension of complexes}\label{Quasi-projective-Properties}
Throughout this section, $R$ is an associated ring with identity. In this section, we introduce the definition of quasi-projective dimensions of complexes and discuss some properties of quasi-projective
dimension.

For a homologically bounded above complex \(M\) and a non-zero \(R\)-module \(N\), set
$$\epsilon_{R}(M,N)=\sup \bigl\{\{0\} \cup \{n\geq1\hspace{0.03cm}|\hspace{0.03cm}\mathrm{Ext}_{R}^{\mathrm{hsup}M + n}(M,N)\neq 0\}\bigr\}.$$
Then \(\epsilon_{R}(M,N)\) takes values in $\mathbb{Z}_{\geq0}\cup\{\infty\}$; in particular it is zero when all the displayed Ext groups
vanish.

\begin{df}\label{lem:1.1}{\rm  Let \(R\) be an associative ring with identity and let \(0\not\simeq M\in \mathrm{D}(R)\).

(1) A \emph{quasi-projective resolution} of \(M\) is a semi-projective complex \(P\) for which there are nonnegative integers \((a_{i})_{i\in \mathbb{Z}}\), not all zero, with \(a_{i} = 0\) for \(i\ll 0\), such that
\[
\mathrm{H}_{j}(P)\cong \bigoplus_{i\in \mathbb{Z}}\mathrm{H}_{j - i}(M)^{\oplus a_{i}}\quad (j\in \mathbb{Z}).
\]
The resolution is \emph{bounded above} if \(\sup P< \infty\).

(2) Such a resolution is \emph{identifying} if \(\mathrm{hsup}M,\mathrm{hsup}P\in \mathbb{Z}\) and
\[
\epsilon_{R}(P,N) = \epsilon_{R}(M,N)
\]
for every non-zero \(R\)-module \(N\) with \(\epsilon_{R}(M,N)< \infty\).

Write \(\mathcal{Q}_{R}(M)\) for the class of bounded above identifying quasi-projective resolutions of \(M\), and define the \emph{quasi-projective dimension} of $M$ as follows:
\[
\mathrm{qpd}_{R}M = \inf_{P\in\mathcal{Q}_{R}(M)}\{\mathrm{sup}P-\mathrm{hsup} P\}+\mathrm{hsup}M.
\]
We use \(\inf \emptyset = \infty\), and set \(\mathrm{qpd}_{R}M =-\infty\) if \(M\simeq 0\).}
\end{df}

The next lemma offers a criterion for quasi-projective dimension.

\begin{lem}\label{lem0.7} 
Let $0\not\simeq M\in\mathrm{D}(R)$ and $g\geq0$. Put $h=\mathrm{hsup}M$.  Then
\(\mathrm{qpd}_R M-h\leq g\) if and only if there is \(P \in \mathcal{Q}_R(M)\)
with \(\mathrm{pd}_R\mathrm{C}_{\mathrm{hsup} P}(P) \leq g\). Moreover, if
\(\mathrm{qpd}_R M < \infty\), then $P$ satisfies
\[
\mathrm{pd}_R\mathrm{C}_{\mathrm{hsup} P}(P) = \sup P - \mathrm{hsup} P = \mathrm{qpd}_R M-h.
\]
\end{lem}
\begin{proof} If \(\mathrm{qpd}_R M-h \leq g\), then there is $P\in\mathcal{Q}_{R}(M)$ such that $\mathrm{qpd}_RM-h=\mathrm{sup}P-\mathrm{hsup}P\leq g$. Set $s=\mathrm{sup}P$ and $r=\mathrm{hsup}P$, the exact sequence $$0\rightarrow P_s\rightarrow \cdots\rightarrow P_r\rightarrow\mathrm{C}_{r}(P)\rightarrow0$$ yields
 $\mathrm{pd}_R\mathrm{C}_{r}(P)\leq s-r\leq g$. Conversely,
let $P\in\mathcal{Q}_{R}(M)$ with $\mathrm{pd}_R\mathrm{C}_{\mathrm{hsup}P}(P)\leq g$. Set $r=\mathrm{hsup}P$ and $P_{\subset_{r}}:0\rightarrow\mathrm{C}_{r}(P)\stackrel{\bar{d}_r}\rightarrow P_{r-1}\stackrel{d_{r-1}}\rightarrow\cdots$. Let $0\rightarrow Q_g\rightarrow \cdots\rightarrow Q_0\stackrel{\pi}\rightarrow\mathrm{C}_{r}(P)\rightarrow0$ be a projective resolution of $\mathrm{C}_{r}(P)$, and set $$\bar{P}:0\longrightarrow Q_g\longrightarrow \cdots\longrightarrow Q_0\stackrel{\bar{d}_r\pi}\longrightarrow P_{r-1}\stackrel{d_{r-1}}\longrightarrow\cdots.$$ Then $\bar{P}\in\mathcal{Q}_{R}(M)$, and so $\mathrm{qpd}_RM-h\leq\mathrm{sup}\bar{P}-\mathrm{hsup}\bar{P}\leq g$.
\end{proof}

\begin{rem}\label{lem:1.0}{\rm (1) Let $0\not\simeq M\in\mathrm{D}(R)$. If \(P \in \mathcal{Q}_{R}(M)\), then the multiplicity family has finite nonempty support \(I\). Indeed, set \(h = \mathrm{hsup} M\), \(a_{i} > 0\) implies \(h + i \leq \mathrm{hsup} P\). If \(b = \max I\), then
\(
\mathrm{hsup} P = h + b\) and  \(\mathrm{H}_{h + b}(P) \cong \mathrm{H}_{h}(M)^{\oplus a_{b}}.
\)
If \(M\) is homologically bounded, then so is \(P\).

(2) Let $0\not\simeq M\in\mathrm{D}(R)$. For a semi-projective complex \(P\) with \(h = \mathrm{hsup} P\) and every \(j > 0\), it follows from \cite[Lemma 8.1.6]{CFH} that
\[
\mathrm{Ext}_{R}^{j}(C_{h}(P),N)\cong \mathrm{Ext}_{R}^{h + j}(P,N).
\]
Thus \(
\epsilon_{R}(P,N) = \sup \{j\geq 0\hspace{0.03cm}|\hspace{0.03cm}\mathrm{Ext}_{R}^{j}(C_{h}(P)\oplus R,N)\neq 0\}
\)  for a nonzero $R$-module \(N\).
  Consequently, every semi-projective resolution \(P\) of $M$ is identifying, so $\mathrm{qpd}_RM\leq\mathrm{pd}_RM$.

    (3) Let $0\not\simeq M\in\mathrm{D}_\sqsubset(R)$ and \(P \in \mathcal{Q}_{R}(M)\). For \(s\in\mathbb{Z}\) and a nonzero $R$-module $N$, since \(\epsilon_{R}(\Sigma^{s} M, N) = \epsilon_{R}(M, N)\), it implies that \(P \in \mathcal{Q}_{R}(\Sigma^{s}M)\). Hence \(\mathrm{qpd}_{R}\Sigma^{s}M=\mathrm{qpd}_{R}M+s\).

 (4) Let $M\in\mathrm{D}(R)$ and \(P\) be an identifying quasi-projective resolution of \(M\), and let
 $Q$ be a semi-projective complex with $Q\simeq P$. Then $\mathrm{C}_{\mathrm{hsup}P}(P)\oplus W\cong\mathrm{C}_{\mathrm{hsup}Q}(Q)\oplus W'$ for some projective $R$-modules $W,W'$ by \cite[Lemma 8.1.12]{CFH}, it follows from (2) that \(Q\) is an identifying quasi-projective resolution of \(M\).

 (5) Let $L$ be an $R$-module and $P$ a quasi-projective resolution of $L$ with $\mathrm{hsup}P\in \mathbb{Z}$. For an $R$-module $N$, if $0\leq d=\mathrm{sup}\{i\geq0\hspace{0.03cm}|\hspace{0.03cm}\textrm{Ext}^i_R(L,N)\neq0\}<\infty$, then $d=\mathrm{sup}\{i\geq0\hspace{0.03cm}|\hspace{0.03cm}
 \textrm{Ext}^i_R(\mathrm{C}_{\mathrm{hsup}P}(P),N)\neq0\}$   by \cite[Proposition 2.2(b)]{CCL},
and hence \(P\) is identifying.  Consequently, our definition of quasi-projective dimension extends that of \cite[Definition 3.1]{GJT} from modules to derived categories, and the two notions agree when restricted to \(R\)-modules.

(6) Let \(R = k[[x,y]]\) and \(M = k\oplus \Sigma R\). The complex \(P = (0\rightarrow R^{2}\xrightarrow{(x,y)}R\rightarrow 0)\), in degrees \(1,0\), has the same homology as \(M\), but it is not identifying since \(\epsilon_{R}(P,k) = 0\), whereas \(\epsilon_{R}(M,k) = 1\). The semi-projective resolution \(\Sigma R\oplus K(x,y;R)\) of $M$ is in $\mathcal{Q}_{R}(M)$ by (2), and so $\mathrm{qpd}_RM\leq\mathrm{sup}Q-\mathrm{hsup}Q+\mathrm{hsup}M=2$.}
\end{rem}

\begin{lem}\label{lem0.6'} Let \(M\in\mathrm{D}_\Box(R)\) with $\mathrm{qpd}_RM<\infty$.  Then there is a bounded
identifying complex of projective modules $P$ such that
\[\mathrm{qpd}_RM=\sup P-\mathrm{hsup}P+\mathrm{hsup}M.\]
\end{lem}
\begin{proof} Set $h=\mathrm{hsup}M$. As $\mathrm{qpd}_RM<\infty$, there exists $Q\in\mathcal{Q}_{R}(M)$, such that $\mathrm{qpd}_RM =\mathrm{pd}_RC_{\mathrm{hsup}Q}(Q)+h=\sup Q-\mathrm{hsup}Q+h$. Its homology is bounded
by Remark \ref{lem:1.0}(1). Set \(s = \sup Q\). Choose a semi-projective resolution \(F \simeq Q\) with
\(F_i = 0\) for \(i < \mathrm{hinf} Q\). Since \(Q_{s+1} = 0\) and \(s \geq \mathrm{hsup} F\), it follows from \cite[Lemma 8.1.6]{CFH} that
\[
\mathrm{Ext}_R^1(C_s(F),N)\cong\mathrm{Ext}_R^{s+1}(F,N) \cong \mathrm{Ext}_R^{s+1}(Q,N) = 0
\]
for every $R$-module \(N\).
Thus \(C_s(F)\) is projective. The truncation
\[
P = 0 \to C_s(F) \to F_{s-1} \to \cdots \to F_{\mathrm{hinf} Q} \to 0
\]
is a bounded identifying quasi-projective resolution of \(M\) with \(\mathrm{qpd}_RM=\mathrm{pd}_RC_{\mathrm{hsup}P}(P)+h=\sup P-\mathrm{hsup}P+h\) by
\cite[Lemma 8.1.12]{CFH} and  Remark \ref{lem:1.0}(4).
\end{proof}

The following is the main result in this section, which is proved for finitely generated $R$-modules over commutative noetherian rings by Gheibi, Jorgensen and Takahashi in \cite[Corollary 4.10]{GJT}. Recently, Chen, Chen and Liu \cite[Theorem 1.2]{CCL} generalized this result to any object in an abelian category with enough projectives.

\begin{thm}\label{lem0.6} $(1)$ Let $0\not\simeq M\in\mathrm{D}(R)$.
If $\mathrm{pd}_RM<\infty$, then $\mathrm{qpd}_RM=\mathrm{pd}_RM$.

$(2)$ Let \(M\in\mathrm{D}_\Box(R)\). If $\mathrm{qpd}_RM<\infty$ and $\mathrm{Ext}^{j}_R(\mathrm{H}_s(M),\mathrm{H}_t(M))=0$ for all $j\geq2$ and $s,t\in\mathbb{Z}$, then $\mathrm{pd}_RM<\infty$.\\In particular, one has an equality
$$\footnotesize{\mathrm{sup}\{\mathrm{pd}_RM\hspace{0.03cm}|\hspace{0.03cm}M\in\mathrm{D}_\Box(R)\ \textrm{with}\ \mathrm{pd}_RM<\infty\}
=\mathrm{sup}\{\mathrm{qpd}_RM\hspace{0.03cm}|\hspace{0.03cm}M\in\mathrm{D}_\Box(R)\ \textrm{with}\ \mathrm{qpd}_RM<\infty\}}.$$
\end{thm}
\begin{proof} (1) Assume that $\mathrm{qpd}_RM<\mathrm{pd}_RM=n<\infty$. Then $\epsilon_{R}(M,N)=n-\mathrm{hsup}M>0$ for some $R$-module $N$ by \cite[Theorem 8.1.8]{CFH}. On the other hand,
 there exists $P\in\mathcal{Q}_{R}(M)$, such that $\mathrm{qpd}_RM=\mathrm{pd}_R\mathrm{C}_{\mathrm{hsup}P}(P)+\mathrm{hsup}M=\mathrm{sup}P-\mathrm{hsup}P+\mathrm{hsup}M$ by Lemma  \ref{lem0.7}, so $\epsilon_{R}(P,N)=\mathrm{sup}\{i\geq0\hspace{0.03cm}|\hspace{0.03cm}
 \textrm{Ext}^i_R(\mathrm{C}_{\mathrm{hsup}P}(P)\oplus R,N)\neq0\}\leq
 \mathrm{sup}P-\mathrm{hsup}P<n-\mathrm{hsup}M$ by \cite[Theorem 8.1.8]{CFH} again,  a contradiction, it implies that
 $\mathrm{qpd}_RM=\mathrm{pd}_RM$.

(2) Set \(r = \mathrm{hsup} M\) and $M_{\subset r-1}:0\rightarrow\mathrm{C}_{r-1}(M)\rightarrow M_{r-2} \rightarrow\cdots$, the
truncation triangle is
\[
\Sigma^r \mathrm{H}_r(M) \rightarrow M \rightarrow M_{\subset r-1} \rightarrow \Sigma^{r+1}\mathrm{H}_r(M).
\]
By induction, \(M_{\subset r-1}\simeq\bigoplus_{i<r} \Sigma^i \mathrm{H}_i(M)\). Since
\[
\mathrm{Hom}_{\mathrm{D}(R)}(\Sigma^i \mathrm{H}_i(M), \Sigma^{r+1}\mathrm{H}_r(M)) \cong \mathrm{Ext}_R^{r+1-i}(\mathrm{H}_i(M), \mathrm{H}_r(M)) = 0,
\]
the triangle splits. Thus \(M\simeq\bigoplus_{i}\Sigma^i \mathrm{H}_i(M)\).
By Lemma \ref{lem0.6'}, we choose a
 bounded identifying quasi-projective $P$ resolution of \(M\) with finite support $I$, such that \(\mathrm{qpd}_RM=\sup P-\mathrm{hsup}P+r\) and
$\mathrm{H}(P) \cong \bigoplus_{i\in I} \mathrm{H}(\Sigma^i M)^{a_i}$.
 Set $\bar{M}=\bigoplus_{i\in I}\Sigma^{i}M^{a_i}$ and \(b = \max I\). Then $P\simeq\mathrm{H}(P)\simeq\mathrm{H}(\bar{M})\simeq \bar{M}$ by the preceding proof, it implies that \(\mathrm{pd}_RM+b=\mathrm{pd}_R\bar{M}=\sup P=\mathrm{qpd}_RM+\mathrm{hsup}P-r\). Therefore, $\mathrm{qpd}_RM=\mathrm{pd}_RM$ by  Remark \ref{lem:1.0}(1).
\end{proof}

Let $A$ be an $R$-module. The flat dimension and FP-injective dimension of $A$ are denoted by $\mathrm{fd}_RA$ or $\mathrm{FP}\textrm{-}\mathrm{id}_RA$, respectively. The next corollary  is used to further discuss the relationship between quasi-projective dimension and projective dimension.

\begin{cor}\label{lem0.01} Let $P$ be a semi-projective $R$-complex such that $\mathrm{H}(P)\cong\mathrm{H}(M)$ for some $R$-complex $M$.

$(1)$ If $R$ is left coherent and $P$ is perfect such that either $\mathrm{fd}_R\mathrm{H}_t(P)\leq1$ or $\mathrm{FP}\textrm{-}\mathrm{id}_R\mathrm{H}_t(P)\leq1$ for all $t\in\mathbb{Z}$, then $P\stackrel{\simeq}\rightarrow M$.

 $(2)$  If $R$ is von Neumann regular and $P$ is perfect, then $P\stackrel{\simeq}\rightarrow M$.
\end{cor}
\begin{proof} (1) As $R$ is left coherent, each $\mathrm{H}_s(P)$ is finitely presented. If $\mathrm{fd}_R\mathrm{H}_t(P)\leq1$ for all $t\in\mathbb{Z}$, then $\mathrm{Ext}^{p}_{R}(\mathrm{H}_{q}(\Sigma^{-\mathrm{inf}P}P),\mathrm{H}_{s-q}(\Sigma^{-\mathrm{sup}P}M))^+
\cong\mathrm{Tor}^R_{p}(\mathrm{H}_{s-q}((\Sigma^{-\mathrm{sup}P}M)^+),\mathrm{H}_q(\Sigma^{-\mathrm{inf}P}P))=0$ for $p\geq2$. If $\mathrm{FP}\textrm{-}\mathrm{id}_R\mathrm{H}_{t}(M)\leq1$ for all $t\in\mathbb{Z}$, then $\mathrm{Ext}^{p}_{R}(\mathrm{H}_{s}(\Sigma^{-\mathrm{inf}P}P),\mathrm{H}_{s-q}(\Sigma^{-\mathrm{sup}P}M))=0$ for $p\geq2$. Thus  there is a quasi-isomorphism $P\stackrel{\simeq}\rightarrow M$ by Theorem \ref{lem0.6}(2).

(2)  Since $P$ is perfect, it follows from \cite[Theorem 1.3]{HLP} that $$\mathrm{H}_0(\mathrm{Hom}_R(P,M))\stackrel{\cong}\rightarrow
\prod_{s\in\mathbb{Z}}\mathrm{Hom}_R(\mathrm{H}_s(P),\mathrm{H}_s(M)).$$  Thus
there is a chain map $f:P\rightarrow M$ such that $\mathrm{H}_s(f)$ is an isomorphism for all $s\in\mathbb{Z}$.
\end{proof}

\begin{prop}\label{lem0.14} Let $0\not\simeq M,N\in\mathrm{D}(R)$ with $\mathrm{hsup}M=\mathrm{hsup}N$. There exists an inequality
$$\mathrm{qpd}_R(M\oplus N)\leq\mathrm{max}\{\mathrm{qpd}_RM,\mathrm{qpd}_RN\}.$$
\end{prop}
\begin{proof} We may assume that $\mathrm{max}\{\mathrm{qpd}_RM,\mathrm{qpd}_RN\}<\infty$.  Put \(h = \mathrm{hsup}M = \mathrm{hsup}N\), and let $P\in\mathcal{Q}_{R}(M)$ and $Q\in\mathcal{Q}_{R}(N)$ with finite supports \(I, J\) and multiplicities \(a_i, b_j > 0\), such that $\mathrm{qpd}_RM=\mathrm{sup}P-\mathrm{hsup}P+h$, $\mathrm{qpd}_RN=\mathrm{sup}Q-\mathrm{hsup}Q+h$ and
\[
\mathrm{H}(P) \cong \bigoplus_{i \in I} \mathrm{H}(\Sigma^i M)^{a_i}, \quad
\mathrm{H}(Q) \cong \bigoplus_{j \in J} \mathrm{H}(\Sigma^j N)^{b_j}.
\]Put \(\alpha = \max I\), \(\beta = \max J\) and
\[
F = \bigoplus_{j \in J} (\Sigma^j P)^{b_j} \oplus \bigoplus_{i \in I} (\Sigma^i Q)^{a_i}.
\]
Then \(F\) is semi-projective and
\[
\mathrm{H}_k(F)\cong \bigoplus_{i\in I,\, j\in J}\mathrm{H}_{k - i - j}(M\oplus N)^{a_i b_j}.
\]In particular \(\mathrm{hsup} F = h + \alpha +\beta = :r\), and
\[
\sup F - r = \max \{\sup P - (h + \alpha),\ \sup Q - (h + \beta)\}.
\]
Let \(L\) be a non-zero $R$-module with \(\epsilon_{R}(M\oplus N,L)< \infty\). Then \(\epsilon_{R}(M,L)\) and \(\epsilon_{R}(N,L)\) are finite. For \(j\in J\), $\mathrm{Ext}_{R}^{r + n}(\Sigma^{j}P,L)\cong\mathrm{Ext}_{R}^{r + n-j}(P,L)$ gives
\[\mathrm{sup} \bigl(\{0\} \cup \{n\in\mathbb{Z}\hspace{0.03cm}|\hspace{0.03cm}\mathrm{Ext}_{R}^{r + n}(\Sigma^{j}P,L)\neq 0\}\bigr)
= \mathrm{sup} \{\epsilon_{R}(P,L) - (\beta -j),\ 0\}.
\]
These numbers are at most \(\epsilon_{R}(P,L)\), and equality holds for \(j = \beta\). The analogous statement for \(\Sigma^{i}Q\) attains \(\epsilon_{R}(Q,L)\) at \(i = \alpha\), it implies that $$\epsilon_{R}(F,L) = \max \{\epsilon_{R}(P,L),\epsilon_{R}(Q,L)\}
= \max \{\epsilon_{R}(M,L),\epsilon_{R}(N,L)\}
= \epsilon_{R}(M\oplus N,L).$$
Thus \(F\) is identifying, and
 $\mathrm{qpd}_R(M\oplus N)\leq\mathrm{sup}F-\mathrm{hsup}F+h\leq\mathrm{max}\{\mathrm{sup}P-\mathrm{hsup}P+h,
 \mathrm{sup}Q-\mathrm{hsup}Q+h\}
 =\mathrm{max}\{\mathrm{qpd}_RM,\mathrm{qpd}_RN\}$.
\end{proof}

\bigskip
\section{\bf The derived Auslander-Buchsbaum formula}\label{AB-formula}
In this section, a derived Auslander-Buchsbaum formula is established for complexes of finite quasi-projective dimension over local noetherian rings. This result extends the classical Auslander-Buchsbaum formula by replacing projective dimension with quasi-projective dimension and yields several consequences concerning depth and homological bounds. We begin with the following lemmas.

\begin{lem}\label{lem0.0} Let $R$ be left coherent, and let $\mathrm{D}_\Box^\mathrm{f.p}(R)$ denote the full subcategory of $\mathrm{D}(R)$ consisting of $R$-complexes with bounded and finitely presented homology. If $0\not\simeq M\in\mathrm{D}^\mathrm{f.p.}_\Box(R)$  with $\mathrm{qpd}_RM<\infty$, then there is a
perfect complex $P$ that is an identifying quasi-projective resolution of $M$ such that $\mathrm{qpd}_RM =\mathrm{sup}P-\mathrm{hsup}P+\mathrm{hsup}M$. Moreover, if $R$ is local noetherian then one can choose $P$ to be minimal.
\end{lem}
\begin{proof} Let $Q$ be a bounded identifying quasi-projective resolution of $M$, such that $\mathrm{qpd}_RM =\mathrm{sup}Q-\mathrm{hsup}Q+\mathrm{hsup}M$. Then each $\mathrm{H}_i(Q)$ is finitely presented, so there is a quasi-isomorphism $\alpha:F\rightarrow Q$, where $F$ is a complex of finitely generated projective
$R$-modules with $\mathrm{inf}F=\mathrm{inf}Q$. Note that $F,Q$ are semi-projective, $\mathrm{cn}(\alpha)$ is contractible, so
each $\mathrm{Z}_i(\mathrm{cn}(\alpha))$ and $\mathrm{B}_i(\mathrm{cn}(\alpha))$ are projective. Set $s=\mathrm{sup}Q$ and $h=\mathrm{hsup}Q$. As $\mathrm{cn}(\alpha)_{i+1}=F_i$ for $i > s$, it implies that $F_s/\mathrm{B}_s(F)\cong\mathrm{B}_s(\mathrm{cn}(\alpha))$ is projective, one get  a perfect complex
$$P:\cdots\rightarrow0\longrightarrow F_s/\mathrm{B}_s(F)\stackrel{\bar{d}_s}\longrightarrow F_{s-1}\stackrel{d_{s-1}}\longrightarrow\cdots\stackrel{d_{\mathrm{inf}Q+1}}\longrightarrow F_{\mathrm{inf}Q}\longrightarrow0\rightarrow\cdots$$
which is quasi-isomorphic to $Q$.
 Hence $P\in\mathcal{Q}_{R}(M)$ by Remark \ref{lem:1.0}(4), and then $\mathrm{qpd}_RM =\mathrm{sup}P-\mathrm{hsup}P+\mathrm{hsup}M$.

If $R$ is local noetherian, then $P=P'\oplus P''$ such that $P'$ and $P''$ are complexes of finitely generated free $R$-modules with $P'$ minimal and $P''$ contractible by \cite[Proposition 1.1.2(iv)]{A}. As $\mathrm{C}_{n}(P)=\mathrm{C}_{n}(P')\oplus\mathrm{C}_{n}(P'')$ and $\mathrm{C}_{n}(P'')$ is a projective $R$-module  for $n\in\mathbb{Z}$, it implies that
$P'$ is minimal and $P'\in\mathcal{Q}_{R}(M)$ by Remark \ref{lem:1.0}(4), and so $\mathrm{qpd}_RM\leq\mathrm{sup}P'-\mathrm{hsup}P'+\mathrm{hsup}M=\mathrm{sup}P-\mathrm{hsup}P+\mathrm{hsup}M=\mathrm{qpd}_RM$, as desired.
\end{proof}

The next theorem proves the Auslander-Buchsbaum formula for complexes of finite quasi-projective dimension, which extends the module case of  \cite[Theorem 4.4]{GJT}.

\begin{thm}\label{lem0.3} Let $(R,\mathfrak{m},k)$ be a local noetherian ring and $M\in\mathrm{D}^\mathrm{f}_\Box(R)$ with $\mathrm{qpd}_RM<\infty$. If either $\mathrm{pd}_RM<\infty$, or $\mathrm{depth}_R\mathrm{H}_{\mathrm{hsup}M}(M)\leq\mathrm{depth}_R\mathrm{H}_i(M)$ for all $i\leq\mathrm{hsup}M$, then
$$\mathrm{qpd}_RM=\mathrm{depth}R-\mathrm{depth}_RM.$$
\end{thm}
\begin{proof} If $\mathrm{pd}_RM<\infty$, then $\mathrm{qpd}_RM=\mathrm{pd}_RM$ by Theorem \ref{lem0.6}(1). Hence \cite[Corollary 2.2]{Iy} yields the desired equality.

  Set $h=\mathrm{hsup}M$ and $u=\operatorname{depth}_R\mathrm{H}_h(M)$. Assume $u\leq\mathrm{depth}_R\mathrm{H}_i(M)$ for all $i\leq h$. Then $u-h\leq\mathrm{depth}_R\mathrm{H}_i(M)-i$,  it follows from \cite[Theorem 2.3]{Iy} that $\mathrm{depth}_RM=u-h$.
Applying Lemma \ref{lem0.0}, we can choose a minimal and perfect identifying quasi-projective resolution $F$ with finite support $I$, such that
$\mathrm{H}(F) \cong \bigoplus_{i \in I} \mathrm{H}(\Sigma^i M)^{a_i}$ and $\mathrm{qpd}_RM =s-r+h$, where $r=\mathrm{hsup}F$ and $s=\sup F$. Then for $j\in\mathbb{Z}$, one has $$\operatorname{depth}_R\mathrm{H}_j(F)-j=\min\{ \operatorname{depth}_R\mathrm{H}_{j-\ell}(M)-(j-\ell)-\ell \hspace{0.03cm}|\hspace{0.03cm}\ell \in I, \mathrm{H}_{j-\ell}(M) \neq 0\}.$$
Set \(b = \max I\), then \(r= h + b\) and  \(\mathrm{H}_{r}(F) \cong \mathrm{H}_{h}(M)^{\oplus a_{b}},\) so $\mathrm{depth}_R\mathrm{H}_{j}(F)\geq u=\mathrm{depth}_R\mathrm{H}_r(F)$ for all $j\leq r$. Hence $\mathrm{depth}_RF=\operatorname{depth}_R\mathrm{H}_r(F)-r=u-r$ by \cite[Theorem 2.3]{Iy} again, it follows from
 \cite[Theorem 2.1]{Iy}  that
\[
\operatorname{depth}_RF
=\operatorname{depth}R-\mathrm{hsup}(k\otimes_R F).
\] Finally, minimality of $F$ gives $\mathrm{hsup}(k\otimes_R F)=s$, it implies that
\(
\mathrm{qpd}_RM=s-r+h
=\operatorname{depth}R-u+h
=\operatorname{depth}R-\operatorname{depth}_RM.
\)
\end{proof}

\begin{cor}\label{lem0.2} Let $(R,\mathfrak{m},k)$ be a local noetherian ring and $M\in\mathrm{D}^\mathrm{f}_\Box(R)$ with $\mathrm{qpd}_RM<\infty$. If $\mathrm{depth}_R\mathrm{H}_{\mathrm{hsup}M}(M)\leq\mathrm{depth}_R\mathrm{H}_i(M)$ for all $i\leq\mathrm{hsup}M$,
then $$\mathrm{depth}_RM+\mathrm{hsup}M\leq\mathrm{depth}R.$$
\end{cor}

\begin{cor}\label{lem0.8}  Let $R$ be a local noetherian ring and $M\in\mathrm{D}^\mathrm{f}_\Box(R)$ with $\mathrm{qpd}_RM<\infty$. If $\mathrm{depth}_R\mathrm{H}_{\mathrm{hsup}M}(M)\leq\mathrm{depth}_R\mathrm{H}_i(M)$ for all $i\leq\mathrm{hsup}M$, then $$\mathrm{qpd}_RM\leq\mathrm{qpd}_R\mathrm{H}_{\mathrm{hsup}M}(M)+\mathrm{hsup}M.$$
\end{cor}
\begin{proof} Set $h=\mathrm{hsup}M$. As $\mathrm{depth}_R\mathrm{H}_{h}(M)\leq\mathrm{depth}_R\mathrm{H}_i(M)$, it follows from \cite[Theorem 2.3]{Iy} that 
 $\mathrm{depth}_RM=\mathrm{depth}_R\mathrm{H}_h(M)-h$. The case of $\mathrm{qpd}_R\mathrm{H}_h(M)=\infty$ is trivial. Assume that $\mathrm{qpd}_R\mathrm{H}_h(M)<\infty$. It follows from Theorem \ref{lem0.3} and \cite[Theorem 4.4]{GJT} that $\mathrm{qpd}_RM=\mathrm{depth}R-\mathrm{depth}_RM=\mathrm{qpd}_R\mathrm{H}_h(M)+h$.
\end{proof}

The following remark gives a partial answer of \cite[Remark 2.7]{G}.

\begin{rem}\label{lem:0.9}{\rm (1) Let $(R,\mathfrak m)$ be noetherian local and let
$0\not\simeq M\in\mathrm D_\Box^{\mathrm f}(R)$ with
$\mathrm{qpd}_RM<\infty$. Put
$\widehat M=\widehat R\otimes_R^{\mathrm{L}}M$ {and $h=\operatorname{hsup}M$}. Then
\[
\mathrm{qpd}_{\widehat R}\widehat M\leq\mathrm{qpd}_RM.
\]
Indeed, choose a minimal and perfect identifying quasi-projective resolution $P$ of $M$ by
Lemma~\ref{lem0.0}. Faithful flatness preserves its homology
decomposition and the two homology suprema. For every
$\widehat R$-module $N$, derived adjunction gives
\[
\operatorname{Ext}^j_{\widehat R}(\widehat M,N)\cong\operatorname{Ext}^j_R(M,N),\quad
\operatorname{Ext}^j_{\widehat R}(\widehat R\otimes_RP,N)
\cong\operatorname{Ext}^j_R(P,N).
\]
Thus $\widehat R\otimes_RP$ is identifying, so
$\mathrm{qpd}_{\widehat R}\widehat M\leq\sup P-\mathrm{hsup}P+h=\mathrm{qpd}_RM$.
If, in addition,
\(
\operatorname{depth}_R\mathrm H_h(M)
\leq\operatorname{depth}_R\mathrm H_i(M)\) for \(i\leq h=\mathrm{hsup}M$,
then
\[
\mathrm{qpd}_{\widehat R}\widehat M=\mathrm{qpd}_RM.
\] by
Theorem~\ref{lem0.3} and \cite[Theorem 2.3]{Iy}.

(2) For a general $M\in\mathrm D_\Box^{\mathrm f}(R)$, the identity
\[
\operatorname{depth}_RM=
\min_i\{\operatorname{depth}_R\mathrm H_i(X)-i\}
\]
need not hold. Let $R=k[[x,y]]$ and let
$F=(0\to R^2\xrightarrow{(x\ y)}R\to0)$ in degrees $1,0$.
Then $\mathrm H_0(F)=k$ and $\mathrm H_1(F)\cong R$, so the
displayed minimum equals $\min\{0,2-1\}=0$.
However, $F$ is a minimal perfect complex with $\sup F=1$;
the classical Auslander-Buchsbaum equality for complexes gives
$\operatorname{depth}_RF=2-1=1$.
Thus a depth computation based only on the homology modules
needs an additional hypothesis. The sufficient condition provides exactly the computation
used in Theorem~\ref{lem0.3}.

(3) Let \(R = k[[x,y]]\) and \(M = k\oplus \Sigma R\). Then $\mathrm{qpd}_RM=2=\mathrm{pd}_RM$. But $\mathrm{depth}_R\mathrm{H}_{1}(M)\nleq\mathrm{depth}_R\mathrm{H}_0(M)$, and $2\nleq\mathrm{qpd}_R\mathrm{H}_{1}(M)+1$, where $h=\mathrm{hsup}M=1$ and $\mathrm{qpd}_R\mathrm{H}_{1}(M)=0$. Thus the assumption in Corollary \ref{lem0.8} cannot be dropped. However, the Auslander-Buchsbaum formula for $M$ remains valid.}
\end{rem}
\bigskip

\section{\bf Quasi-projective dimensions related to complete intersection rings}\label{partial-ans-Q1}
 This section gives some characterizations of complete intersection local rings, and provides some partial answers to Question 1 raised in the introduction.

Recall that the \emph{embedding dimension} of a local noetherian ring $(R,\mathfrak{m},k)$, denoted by $\mathrm{edim}R$, is the least number of generators
of $\mathfrak{m}$ or, equivalently, the dimension of $\mathfrak{m}/\mathfrak{m}^{2}$ as an $R/\mathfrak{m}$-vector space. An ideal $\mathfrak{a}$ is called \emph{$\mathfrak{m}$-primary} if $\mathfrak{m}=\sqrt{\mathfrak{a}}$ the radical of $\mathfrak{a}$.
Next, we answer  \cite[Question 3.12]{GJT} under some conditions.

\begin{thm}\label{lem3.004} Let $(R,\mathfrak{m},k)$ be a local noetherian ring with $\widehat{R}\cong Q/I$ where $(Q,\mathfrak{n})$ is a regular local ring. Suppose one of the following conditions is satisfied

$(1)$ $\mathrm{Thick}_R[M]=\mathrm{Thick}_R[\mathrm{H}(M)]$ for all $M\in\mathrm{D}^\mathrm{f}_\Box(R)\backslash\mathrm{perf}R$.

$(2)$ $\mathrm{dim}Q\leq \mathrm{dim}R+2$.

$(3)$ $I$ is a Burch ideal of $Q$.

$(4)$ the minimal number of generators of $I$ is at most $2$.

$(5)$ the maximal ideal $\mathfrak{m}$ is decomposable. \\
If every finitely generated $R$-module has finite quasi-projective dimension, then $R$ is a complete intersection.
\end{thm}
\begin{proof} (1)  Let $0\not\simeq M\in\mathrm{D}^\mathrm{f}_\Box(R)$. If $M$ is perfect, then $M$ is virtually small. Assume that $M$ is not perfect. For $i\in\mathbb{Z}$, as $\mathrm{qpd}_R\mathrm{H}_i(M)<\infty$, there exists a perfect complex $P$ such that $P\in\mathrm{Thick}_R[\mathrm{H}_i(M)]\subseteq\mathrm{Thick}_R[\mathrm{H}(M)]$, so $\mathrm{H}(M)$ is virtually small. Hence $M$ is  virtually small, and $R$ is a complete intersection.

(2) Since $\mathrm{dim}Q\leq \mathrm{dim}R+2$, the height of $I$ is at most $2$. By \cite[Theorem 6.5]{GJT}, $R$ is Gorenstein, it follows from \cite[Corollary 21.20]{D} that $R$ is a complete intersection.

(3) By \cite[Remark 2.9]{DKT}, $\mathrm{depth}\widehat{R}=0$.
By \cite[Theorem 6.5]{GJT}, $R$ is Gorenstein, it follows from  \cite[Definition 2.8 and Proposition 5.1]{DKT} that $R$ is a complete intersection.

(4) By \cite[Theorem 6.5]{GJT}, $R$ is Gorenstein, so is $\widehat{R}$. If the minimal number of generators of $I$ is at most $2$, it follows by \cite[Remarks 8.11 and 9.3]{T23} that $\widehat{R}$ is a complete intersection. Thus $R$ is a complete intersection.

(5) This follows from \cite[Theorem 6.5]{GJT} and \cite[Corollary 2.7]{NSTV}.
\end{proof}

For a homologically bounded complex $M$,
$M\in\mathrm{Thick}_R[\mathrm H(M)]$. The reverse inclusion
requires additional hypotheses, even over a hypersurface.
For example, let $R=k[[x]]/(x^2)$ and let
$P=(0\to R\xrightarrow{x}R\to0)$ in degrees $1,0$.
Then $P$ is perfect and $\mathrm H_0(P)\cong\mathrm H_1(P)\cong k$.
Thus $\mathrm{Thick}_R[P]\subseteq\mathrm{perf}R$, whereas
$k\in\mathrm{Thick}_R[\mathrm H(P)]$ and $k$ is not perfect.
Consequently, these two thick subcategories are different.
The following corollary concerns nonperfect complexes under the
stated ring-theoretic hypotheses

 \begin{cor}\label{lem3.005'} Let $(R,\mathfrak m,k)$ be a dominant local ring with an isolated
singularity. The following conditions are equivalent:

 $(1)$ $R$ is a complete intersection;

 $(2)$ $k\in\mathrm{Thick}_R[M]$ for every
$M\in\mathrm D_\Box^{\mathrm f}(R)\backslash\mathrm{perf}R$.;

 $(3)$ For every $M\in\mathrm D_\Box^{\mathrm f}(R)\backslash\mathrm{perf}R$,
$\mathrm{Thick}_R[M]=
\mathrm{Thick}_R[R/\mathfrak p\mid\mathfrak p\in\mathrm{Supp}_RM]$.\\
  In addition, if $R$ is artinian then the above conditions are equivalent to

$(4)$ $\mathrm{Thick}_R[M]=\mathrm{Thick}_R[k]$ for all $M\in\mathrm{D}^\mathrm{f}_\Box(R)\backslash\mathrm{perf}R$;

$(5)$ For every $M\in\mathrm D_\Box^{\mathrm f}(R)\backslash\mathrm{perf}R$,
$\mathrm{Thick}_R[M]=\mathrm{Thick}_R[\mathrm H(M)]$,
and every finitely generated $R$-module has finite quasi-projective
dimension.
\end{cor}
\begin{proof} $(2)\Rightarrow(3)$ This follows from \cite[Theorem 6.3]{T}.

 $(3)\Rightarrow(1)$ Take a nonzero
$M\in\mathrm D_\Box^{\mathrm f}(R)$. If it is perfect, it is virtually
small. Otherwise $\mathfrak m\in\mathrm{Supp}_RM$, so (3) gives
$k\in\mathrm{Thick}_R[M]$. The nonzero perfect Koszul complex
$K(\mathfrak m;R)$ belongs to $\mathrm{Thick}_R[k]$ by its finite
length homology and truncation triangles. Hence $M$ is virtually
small. Now apply \cite[Theorem 5.2]{P}.

$(1)\Rightarrow(2)$ The regular case is vacuous.
Otherwise a dominant complete intersection is a hypersurface by
\cite[Propositions 5.10(1) and 6.2(3)]{T23}.
Every nonzero homologically finite bounded complex over a complete
intersection is virtually small by \cite[Theorem 5.2]{P}.
For nonperfect $M$, its image in the singularity category is nonzero.
Since the singularity is isolated, \cite[Proposition 7.7]{T} implies
that this image generates the singularity category as a thick
subcategory. Virtual smallness and \cite[Proposition 7.6]{T} then
give $k\in\mathrm{Thick}_R[M]$.

(2) $\Rightarrow$ (3). This follows by \cite[Theorem 6.3]{T}.

(1) $\Rightarrow$ (2). The regular case is vacuous. Otherwise a dominant complete intersection is a
hypersurface  by \cite[Propositions 5.10(1) and 6.2(3)]{T23}. Let $M\in\mathrm{D}^\mathrm{f}_\Box(R)\backslash\mathrm{perf}R$. As $M$ is virtually small, it follows  from \cite[Propositions 7.6 and 7.7]{T} that $k\in\mathrm{Thick}_R[M]$.

 (3) $\Rightarrow$ (1). Let $M\in\mathrm{D}^\mathrm{f}_\Box(R)$. If $M$ is perfect, then $M$ is virtually small. If $M$ is not perfect, then $k\in\mathrm{Thick}_R[M]$ as $\mathfrak{m}\in\mathrm{Supp}_RM$, and so $K(\mathfrak{m},R)\in\mathrm{Thick}_R[M]$ by \cite[3.10]{DGI}. Thus $M$ is virtually small.

Now suppose $R$ is artinian. Every finite module has a finite
composition series with factors $k$. Truncation triangles therefore
give $\mathrm D_\Box^{\mathrm f}(R)=\mathrm{Thick}_R[k]$.
It follows immediately that $(2)\Leftrightarrow(4)$.
If (4) holds, then (2), and hence (1), holds. For every nonperfect $M$,
\[
\mathrm{Thick}_R[k]=\mathrm{Thick}_R[M]
\subseteq\mathrm{Thick}_R[\mathrm H(M)]
\subseteq\mathrm{Thick}_R[k].
\]
An artinian local ring is complete. Consequently (1) and
\cite[Corollary 3.8]{GJT} imply that every finite $R$-module has
finite quasi-projective dimension. This proves $(4)\Rightarrow(5)$.
Finally, $(5)\Rightarrow(1)$ follows from Theorem~\ref{lem3.004}(1)
\end{proof}

The next result follows directly from Theorem \ref{lem3.004}.

 \begin{cor}\label{lem3.004'} Let $R\cong Q/I$ be an artinian ring with $(Q,\mathfrak{n})$ a regular local ring.
If either $\mathrm{edim}R\leq2$, or $I$ is a Burch ideal of $Q$, or the minimal number of generators of $I$ is at most $2$, then $R$ is a complete intersection if and only if $\mathrm{qpd}_{R}M<\infty$ for every finitely generated $R$-module $M$.
\end{cor}

\begin{rem}\label{lem:2.3} {\rm (1)  By \cite[Proposition 8.1]{T23}, there is an artinian dominant local ring $R$ such that $R$ is not a complete intersection.

(2)  Let $R=k[x,y]/(x^2,y^2)$ with $k$ an infinite field. Then $R$ is an artinian complete intersection ring. As $k\not\in\mathrm{Thick}_R[R,R/(x)]$, where $R/(x)$ is not perfect, $R$ is not dominant. Also $\mathrm{Thick}_R[R/(x)]\neq\mathrm{Thick}_R[k\oplus R]$. Thus the assumption in Corollary \ref{lem3.005'} that $R$ is  dominant is
indispensable.

(3) Let $R=k[x,y]/(x^2,xy,y^2)$ with $k$ a field. Then $R$ is artinian with $\mathfrak{m}^2=0$ and $\mathrm{edim}R=2$. As $\mathrm{qpd}_RR/(x)=\infty$ by \cite[Example 4.8]{GJT}, $R$ is not a complete intersection.}
\end{rem}

\bigskip
\section{\bf Quasi-projective dimensions related to regular sequences}\label{par-ans-Q2}
Throughout this section all rings are commutative.
We study the behavior of quasi-projective dimension under quotients by regular sequences, and prove the sharp inequality in Theorem \ref{lem3.002}, which yields a partial answer to Question 2 in the introduction.

\begin{lem}\label{lem:top}
Let \(x\) be an \(R\)-regular element, and let \(M\in \mathrm{D}_{\square}(R / (x))\) and \(N\) be a nonzero \(R / (x)\)-module. For \(A = R,R / (x)\), write
\[
t_{A}(M,N) = \sup \{i\in \mathbb{Z}\mid \operatorname{Ext}_{A}^{i}(M,N)\neq 0\},\quad \sup \emptyset = -\infty .
\]
If \(t_{R / (x)}(M,N)\) is finite, then
\begin{equation}
t_{R}(M,N) = t_{R / (x)}(M,N) + 1. \tag{6.1}
\end{equation}
If all the Ext groups over \(R/(x)\) vanish, then all those over \(R\) vanish. Moreover, if \(\epsilon_{R / (x)}(M,N)< \infty\), then \(\epsilon_{R}(M,N)< \infty\) and
\begin{equation}
\epsilon_{R / (x)}(M,N) = \max \{\epsilon_{R}(M,N) - 1,0\} . \tag{6.2}
\end{equation}
\end{lem}
\begin{proof}
Set \(D = \mathrm{RHom}_{R}(R/(x),N)\). The exact sequence \(0\to R\xrightarrow{x}R\to R / (x)\to 0\) and \(xN = 0\) give \(\mathrm{H}_{0}(D)\cong N\cong \mathrm{H}_{-1}(D)\) as \(R/(x)\)-modules, and all other homology vanishes. Its truncation triangle in \(\mathrm{D}(R/(x))\) is
\[
N\rightarrow D\rightarrow \Sigma^{-1}N\rightarrow \Sigma N.
\]
Apply \(\mathrm{RHom}_{R / (x)}(M, -)\) to the triangle and use isomorphism
\[
\mathrm{RHom}_{R / (x)}(M,\mathrm{RHom}_{R}(R/(x),N))\simeq \mathrm{RHom}_{R}(M,N),
\]
 we get \(\mathrm{RHom}_{R/(x)}(M,N)\to \mathrm{RHom}_{R}(M,N)\to \Sigma^{-1}\mathrm{RHom}_{R/(x)}(M,N)\to \Sigma\mathrm{RHom}_{R/(x)}(M,N)\). If $t=\mathrm{sup}\{i\in\mathbb{Z}\hspace{0.03cm}|\hspace{0.03cm}
\textrm{Ext}^{i}_{R/(x)}(M,N)\neq0\}$, then the long exact sequence implies \(\textrm{Ext}^{t+1}_{R}(M,N)\cong \textrm{Ext}^{t}_{R/(x)}(M,N)\neq 0\) and \(\textrm{Ext}^{j}_{R}(M,N) = 0\) for \(j > t + 1\). If \(\mathrm{RHom}_{R/(x)}(M,N)\simeq 0\), then \(\mathrm{RHom}_{R}(M,N)\simeq 0\). This proves the first assertions. For \(h = \mathrm{hsup}M\) and \(t<\infty\), since
\[
\epsilon_{R/(x)}(M,N) = \max \{t - h,0\} ,\quad \epsilon_{R}(M,N) = \max \{t + 1 - h,0\} .
\]
In the all-zero case both quantities are zero. This  proves the equality (6.2).
\end{proof}

\begin{prop}\label{lem4.002} Let $\textbf{x}=x_1,\cdots,x_d$ be an $R$-regular sequence, and let $M\in\mathrm{D}_\Box(R/(\textbf{x}))$. Then
 $\mathrm{qpd}_{R/(\textbf{x})}M\leq \mathrm{pd}_{R}M-d$.
\end{prop}
\begin{proof} We may assume $\mathrm{pd}_{R}M<\infty$, and set \(h = \mathrm{hsup}M\). Choose a bounded semi-projective resolution \(P\simeq M\) over \(R\) with \(\mathrm{sup}P = \mathrm{pd}_{R}M\), and put \(Q = P\otimes_{R}R/(\emph{\textbf{x}})\). Then \(Q\) is a bounded complex of projective \(R/(\emph{\textbf{x}})\)-modules and \(\mathrm{sup}Q\leq \mathrm{sup}P\). As $\emph{\textbf{x}}$ is $R$-regular, it follows that \(Q\simeq K(\emph{\textbf{x}};P)\simeq K(\emph{\textbf{x}};M).\)
 Since \(\emph{\textbf{x}}M = 0\), \(K(\emph{\textbf{x}};M)\simeq\bigoplus_{i = 0}^{d}\Sigma^{i}M^{\oplus \binom{d}{i}}\). Consequently
\[
\mathrm{H}_{j}(Q)\cong \bigoplus_{i = 0}^{d}\mathrm{H}_{j-i}(M)^{\oplus \binom{d}{i}},\quad \mathrm{hsup}Q = h + d.
\]
As both sides are annihilated by \((\emph{\textbf{x}})\), these homology isomorphisms are \(R/(\emph{\textbf{x}})\)-linear. Let \(N\neq 0\) be an \(R/(\emph{\textbf{x}})\)-module with \(\epsilon_{R/(\emph{\textbf{x}})}(M,N)< \infty\). Iterating (6.1) along the regular sequence shows that, if \(t = t_{R/(\emph{\textbf{x}})}(M,N)\in \mathbb{Z}\), then \(t_{R}(M,N) = t + d\). Since
\[
\mathrm{RHom}_{R/(\emph{\textbf{x}})}(Q,N)\simeq \mathrm{RHom}_{R}(P,N)\simeq \mathrm{RHom}_{R}(M,N),
\]
it follows that \(\epsilon_{R/(\emph{\textbf{x}})}(Q,N) = \max \{t + d - (h + d),0\} = \epsilon_{R/(\emph{\textbf{x}})}(M,N)\) in the finite case. If $t=-\infty$, then $\textrm{Ext}^{i}_{R}(M,N)=0$ for all $i$,
so $\epsilon_R(M,N)=0$ and then $\epsilon_{R/(\emph{\textbf{x}})}(Q,N)=0$. Hence \(Q\in \mathcal{Q}_{R/(\emph{\textbf{x}})}(M)\) and
\(
\mathrm{qpd}_{R/(\emph{\textbf{x}})}M\leq \mathrm{sup}Q - \mathrm{hsup}Q+h\leq \mathrm{pd}_{R}M - d.
\)
\end{proof}

\begin{lem}\label{lem:power-cone-homology}
Let \(x\in R\) and \(P\) be an $R$-complex such that \(x\mathrm{H}_{i}(P) = 0\) for every \(i\). Then for every integer \(n\geq 2\),
\[
\mathrm{H}_{i}(K(x^{n};P))\cong \mathrm{H}_{i}(P)\oplus \mathrm{H}_{i - 1}(P)
\]
as \(R\)-modules.
\end{lem}
\begin{proof} Consider the following commutative diagram
\begin{center} $\xymatrix@C=25pt@R=20pt{
 P \ar[d]\ar[r]^x& P\ar[d]^{x^{n-1}}\\
P \ar[r]^{x^n}& P }$
\end{center}it induces a commutative diagram of mapping cone sequences. On homology it gives the following diagram
\begin{center} $\xymatrix@C=23pt@R=20pt{
 0\ar[r]& \mathrm H_i(P) \ar[d]^0\ar[r]^{\iota_1\ \ \ \ }& \mathrm H_i(K(x;P))\ar[d]^f\ar[r]^{\ \ \ \pi_1} &\mathrm H_{i-1}(P)\ar@{=}[d]\ar[r]&0 \\
0\ar[r]& \mathrm H_i(P)\ar[r]^{\iota_n\ \ \ \ }& \mathrm H_i(K(x^n;P))\ar[r]^{\ \ \ \pi_n} &\mathrm H_{i-1}(P)\ar[r]&0 }$
\end{center}The left vertical map is zero because \(n\geq 2\) and \(x\mathrm{H}_{i}(P) = 0\). Thus \(f\) factors through a map \(\sigma :\mathrm{H}_{i - 1}(P)\to \mathrm{H}_{i}(K(x^{n};P))\). Since \(\pi_{n}f = \pi_{1}\) and \(\pi_{1}\) is surjective, \(\pi_{n}\sigma = 1\). The bottom row therefore splits, as desired.
\end{proof}

\begin{lem}\label{lem:identifying-transfer}
Let \(x\) be an \(R\)-regular element and \(M\in \mathrm{D}_{\square}(R/(x))\). Suppose that \(P\in \mathcal{Q}_{R}(M)\) is bounded and set \(Q = P\otimes_{R}R/(x)\). Then \(Q\) is a bounded complex of projective \(R/(x)\)-modules, \(\mathrm{hsup}Q = \mathrm{hsup}P + 1\) and
\[
\epsilon_{R/(x)}(Q,N) = \epsilon_{R/(x)}(M,N)
\]
for every nonzero \(R/(x)\)-module \(N\) with \(\epsilon_{R/(x)}(M,N)< \infty\). If, in addition, the homology of \(Q\) satisfies Definition 3.1(1) for \(M\) over \(R/(x)\), then \(Q\in \mathcal{Q}_{R/(x)}(M)\).
\end{lem}
\begin{proof} Set \(h = \mathrm{hsup}P\). Since each \(\mathrm{H}_{i}(P)\) is a direct sum of homology modules of \(M\), \(\mathrm{H}_{i}(P)\) is annihilated by \(x\), so the quasi-isomorphism \(Q\simeq K(x;P)\) and the cone sequence give
\[
0\rightarrow \mathrm{H}_{i}(P)\rightarrow \mathrm{H}_{i}(Q)\rightarrow \mathrm{H}_{i - 1}(P)\rightarrow 0.
\]
They imply \(\mathrm{H}_{h+ 1}(Q)\cong \mathrm{H}_{h}(P)\neq 0\) and \(\mathrm{H}_{i}(Q) = 0\) for \(i >h + 1\).
For \(N\) as in the statement, (6.2) gives \(\epsilon_{R}(M,N)< \infty\). Thus \(P\) being identifying implies \(\epsilon_{R}(P,N) = \epsilon_{R}(M,N)\). As \(\operatorname{Ext}_{R/(x)}^{i}(Q,N)\cong \operatorname{Ext}_{R}^{i}(P,N)\) for all \(i\) and \(\operatorname{hsup}Q =h + 1\), the definition of epsilon yields
\[
\epsilon_{R/(x)}(Q,N)
= \max \{\epsilon_{R}(P,N) - 1,0\}
= \max \{\epsilon_{R}(M,N) - 1,0\} = \epsilon_{R/(x)}(M,N).
\]
The final assertion now follows from Definition 3.1.
\end{proof}

The next proposition strengthens \cite[Proposition 2.11]{GJT}.

\begin{prop}\label{lem4.001} Let \(\textbf{x}= x_{1},\ldots ,x_{d}\) be an \(R\)-regular sequence and \(M\in \mathrm{D}_{\square}(R/\textbf{x})\). For every choice of integers \(n_{1},\ldots ,n_{d}\geq 2\),
\[
\mathrm{qpd}_{R/(x_{1}^{n_{1}},\ldots ,x_{d}^{n_{d}})}M\leq \mathrm{qpd}_{R}M - d.
\]
In particular, one may take a common exponent \(n\geq 2\).
\end{prop}
\begin{proof} Assume \(\mathrm{qpd}_{R}M< \infty\). {Put \(h=\operatorname{hsup}M\).} By Lemma \ref{lem0.6'}, choose a bounded \(P_{0}\in \mathcal{Q}_{R}(M)\) such that $\mathrm{qpd}_{R}M=\mathrm{sup}P_0-\mathrm{hsup}P_0+h$. Put \(R_{0} = R\) and, for \(1\leq j\leq d\), set
\[
R_{j} = R / (x_{1}^{n_{1}},\ldots ,x_{j}^{n_{j}}),\quad P_{j} = P_{j - 1}\otimes_{R_{j-1}}R_{j}.
\]
By \cite[Exercises 1.1.10(b)]{BH}, powers of a regular sequence remain regular. In particular \(x_{j}\) and \(x_{j}^{n_{j}}\) are \(R_{j - 1}\)-regular. Suppose inductively \(P_{j - 1}\in \mathcal{Q}_{R_{j - 1}}(M)\). Its homology is a finite direct sum of shifts of \(\mathrm{H}(M)\), so it is annihilated by \(x_{j}\). By Lemma \ref{lem:power-cone-homology} and the Koszul resolution of \(R_{j}\) over \(R_{j - 1}\), one has
\[
\mathrm{H}_{i}(P_{j})\cong \mathrm{H}_{i}(P_{j - 1})\oplus \mathrm{H}_{i - 1}(P_{j - 1}).
\]
This is an \(R_{j}\)-linear isomorphism, it implies the required homology decomposition for \(P_{j}\). Applying Lemma \ref{lem:identifying-transfer} to \(x_{j}^{n_{j}}\), one has \(P_{j}\in \mathcal{Q}_{R_{j}}(M)\). Moreover \(\sup P_{j}\leq \sup P_{j - 1}\) and \(\operatorname{hsup}P_{j} = \operatorname{hsup}P_{j - 1} + 1\). Therefore,
\[
\begin{aligned}
\mathrm{qpd}_{R_{d}}M
&\leq \sup P_{d} - \mathrm{hsup}P_{d}+h\\
&\leq \sup P_{0} - \mathrm{hsup}P_{0}+h - d\\
&= \mathrm{qpd}_{R}M - d.
\end{aligned}
\]the proof is complete.
\end{proof}

 The following theorem answers \cite[Question 2.12]{G} in the special case.

  \begin{thm}\label{lem3.002} Let \(\textbf{x}= x_{1},\ldots ,x_{d}\) be an \(R\)-regular sequence and \(M\in \mathrm{D}_{\square}(R/(\textbf{x}))\) with \(\mathrm{qpd}_{R}M<\infty\). If
\(
\operatorname{Ext}_{R}^{1}(\mathrm{H}_{s}(M),\mathrm{H}_{t}(M)) = 0\ \text{for all } s,t\in \mathbb{Z},
\)
then \(\mathrm{qpd}_{R/(\textbf{x})}M\leq\mathrm{qpd}_{R}M-d\).
\end{thm}
\begin{proof} We first prove the one-element assertion. Let \(x\) be \(R\)-regular and \(M\in \mathrm{D}_{\square}(R/(x))\). {Put \(h=\operatorname{hsup}M\).} Choose a bounded \(P\in \mathcal{Q}_{R}(M)\) with finite support $I$, such that $$\mathrm{qpd}_{R}M=\mathrm{sup}P-\mathrm{hsup}P+h,\ \mathrm{H}(P) \cong \bigoplus_{i \in I} \mathrm{H}(\Sigma^i M)^{a_i}.$$  As $I$ is finite, each \(\mathrm{H}_{i}(P)\) is a finite direct sum of modules \(\mathrm{H}_{s}(M)\), it implies that
\[
\operatorname{Ext}_{R}^{1}(\mathrm{H}_{i - 1}(P),\mathrm{H}_{i}(P)) = 0.
\]
Put \(Q = P\otimes_{R}R/(x)\). Since \(x\mathrm{H}_{i}(P) = 0\), the cone sequences give
\[
0\rightarrow \mathrm{H}_{i}(P)\rightarrow \mathrm{H}_{i}(Q)\rightarrow \mathrm{H}_{i - 1}(P)\rightarrow 0,
\]
they split over \(R\). As all their terms are \(R/(x)\)-modules, so the splittings are \(R/(x)\)-linear. Thus \(\mathrm{H}(Q)\cong \mathrm{H}(P)\oplus \mathrm{H}(\Sigma P)\) as graded \(R/(x)\)-modules, which is the homology decomposition required by Definition 3.1(1). By Lemma \ref{lem:identifying-transfer}, \(Q\in \mathcal{Q}_{R/(x)}(M)\), and hence
\[
\begin{aligned}
\mathrm{qpd}_{R/(x)}M
&\leq \sup Q - \mathrm{hsup}Q+h\\
&\leq \sup P - \mathrm{hsup}P+h - 1\\
&= \mathrm{qpd}_{R}M - 1.
\end{aligned}
\] For $s,t\in\mathbb{Z}$, the exact sequence $$0\rightarrow\mathrm{H}_s(M)\rightarrow E\stackrel{a}\rightarrow\mathrm{H}_t(M)\rightarrow0$$ of $R/(x)$-modules splits as an $R$-module sequence. Its \(R\)-linear section is also \(R/(x)\)-linear. Therefore \(\mathrm{Ext}_{R/(x)}^{1}(\mathrm{H}_{s}(M),\mathrm{H}_{t}(M)) = 0\) for all \(s,t\). Applying the one-element assertion successively over \(R /(x_{1},\ldots ,x_{j})\) proves the desired result.
 \end{proof}

We now state an immediate corollary of \cite[Proposition 3.5]{GJT} and Theorem \ref{lem3.002}.

 \begin{cor}\label{lem4.004} Let $R$ be a commutative noetherian local ring and $\textbf{x}=x_1,\cdots,x_d$ an $R$-regular sequence, and let $M$ be a nonzero finitely generated $R/(\textbf{x})$-module. If
$\mathrm{Ext}^{1}_R(M,M)=0$, then $\mathrm{qpd}_{R}M<\infty$ if and only if $\mathrm{qpd}_{R/(\textbf{x})}M<\infty$.
\end{cor}
\begin{proof}
The forward implication is Theorem~\ref{lem3.002}. For the reverse implication, use the module compatibility in Remark 3.3(5) and apply [19, Proposition 3.5] successively along the regular sequence. This gives
\(
\mathrm{qpd}_{R}M\leq \mathrm{qpd}_{R /(\emph{\textbf{x}})}M + d< \infty .\)
\end{proof}

\begin{prop}\label{lem0.1} Let $M\in\mathrm{D}^\mathrm{f}_\Box(R)$ and $(R,\mathfrak{m},k)$ be a local noetherian ring, and let $x\in \mathfrak{m}$ be regular on both $R$ and each $\mathrm{H}_i(M)$.
 If $P$ is an identifying quasi-projective resolution
of $M$, then $Q=P\otimes_RR/(x)$ is an identifying quasi-projective resolution of $M\otimes^\mathrm{L}_RR/(x)$ over $R/(x)$.
Moreover, $$\mathrm{qpd}_{R/(x)}M\otimes^\mathrm{L}_RR/(x)\leq\mathrm{qpd}_RM.$$
\end{prop}
\begin{proof} As $M\otimes^\mathrm{L}_RR/(x)\simeq K(x;M)$, the cone sequence and regularity on homology give $\mathrm{H}_{i}(M\otimes^\mathrm{L}_RR/(x))\cong\mathrm{H}_{i}(M)/x\mathrm{H}_{i}(M)$. Since the homology is finitely generated and \(x\in \mathfrak{m}\), Nakayama's lemma gives $\mathrm{hsup}(M\otimes^\mathrm{L}_RR/(x))=\mathrm{hsup}M$.  We first consider the given identifying resolution $P$.
There is a finite support $I$, such that $$\mathrm{H}(P)\cong\bigoplus_{i\in I}
\mathrm{H}(\Sigma^{i}M^{a_i}).$$
Derived adjunction implies that $Q$ is
 semi-projective over $R/(x)$. For $j\in\mathbb{Z}$, since $x$ is $\mathrm{H}_j(P)$-regular and $Q\simeq K(x;P)$, the cone sequence implies that $\mathrm{hsup}Q=\mathrm{hsup}P$, and $$\mathrm{H}_j(Q)\cong\mathrm{H}_j(K(x;P))\cong\mathrm{H}_{j}(P)/x\mathrm{H}_{j}(P)
\cong\bigoplus_{i}\mathrm{H}_{j-i}(M)^{a_i}\otimes_RR/(x).$$
Let $N$ be a non-zero $R/(x)$-module with $\epsilon_{R/(x)}(M\otimes^\mathrm{L}_RR/(x),N)<\infty$. Since $$\textrm{RHom}_{R/(x)}(Q,N)\cong\textrm{RHom}_{R/(x)}({R/(x)},\textrm{RHom}_R(P,N))
\cong\textrm{RHom}_R(P,N),$$ $\epsilon_{R}(M,N)<\infty$, so $\epsilon_{R/(x)}(Q,N)=\epsilon_{R}(P,N)=\epsilon_{R}(M,N)=
\epsilon_{R/(x)}(M\otimes^\mathrm{L}_RR/(x),N)$. Hence
 $Q$ is an identifying quasi-projective resolution of $M\otimes^\mathrm{L}_RR/(x)$ over $R/(x)$.  If $\mathrm{qpd}_RM=\infty$, then the dimension inequality
 is immediate. Otherwise choose a minimal perfect
identifying resolution $P$ by Lemma~\ref{lem0.0}.
Then $Q$ is bounded and belongs to $\mathcal Q_{R/(x)}(M\otimes^\mathrm{L}_RR/(x))$, and $\sup Q\le\sup P$. Consequently, $\mathrm{qpd}_{R/(x)}M\otimes^\mathrm{L}_RR/(x)
 \le\sup Q-\mathrm{hsup}Q+\mathrm{hsup}(M\otimes^\mathrm{L}_RR/(x))
 \le\sup P-\mathrm{hsup}P+\mathrm{hsup}M
 =\mathrm{qpd}_RM$.
\end{proof}

\smallskip
{\bf Acknowledgements.} The authors would like to thank Souvik Dey for his valuable comments on Theorems 3.4, 4.2 and 5.1, and for providing the example in Remark \ref{lem:1.0}(6). They also thank Hiroki Matsui for his helpful comments and suggestions on Definition 3.1, Theorem 3.5(1) and Corollary 4.4.
The research was supported by the National Natural Science Foundation of China (Grant 12122112, 12031014, 12571035, 12671053).

\vspace{3mm}
\noindent\textbf{Hongxing Chen}\\
School of Mathematical Sciences \&  Academy for Multidisciplinary Studies, \\Capital Normal University, Beijing 100048, P. R. China.\\
Email: \textsf{chenhx@cnu.edu.cn}\\[1mm]
\textbf{Jiangsheng Hu}\\
School of Mathematics, Hangzhou Normal University, Hangzhou 311121, P. R. China.\\
Email: \textsf{hujs@hznu.edu.cn}\\[1mm]
\textbf{Xiaoyan Yang}\\
School of Science, Zhejiang University of Science and Technology, Hangzhou 310023, P. R. China.\\
Email: \textsf{yangxy@zust.edu.cn}\\[1mm]
\end{document}